\documentclass[final]{siamltex}
\usepackage{graphicx,amssymb}

\def\ak{a_k^{}} \def\bk{b_k^{}} \def\ck{c_k^{}}
\def\akt{\tilde a_k^{}} \def\bkt{\tilde b_k^{}} \def\ckt{\tilde c_k^{}}
\def\cmk{c_{-k}^{}} 
\def\tk{t_k^{}} \def\fk{f_k^{}}
\def\tj{t_j^{}} \def\fj{f_j^{}}
\def\pipi{[\kern .5pt 0,2\pi]}
\def\ones{[-1,1]}
\def\fn{f_n^{}}
\def\pn{p_n^{}}
\def\pns{p_n^*}
\def\Pn{P_n^{}}
\def\IN{I_N^{}}
\def\topspace{\vskip .06in}
\def\bottomspace{\vskip .06in}

\title{Extension of Chebfun to periodic functions}

\author{Grady B. Wright\thanks{Dept. of Mathematics, Boise State University,
Boise, ID 83725-1555, USA}\and
Mohsin Javed\and Hadrien Montanelli\and 
Lloyd N. Trefethen\thanks{Oxford University
Mathematical Institute, Oxford OX2 6GG,
UK.\ \ Supported by the European Research Council
under the European Union's Seventh Framework Programme
(FP7/2007--2013)/ERC grant agreement no.\ 291068.\ \ The views expressed
in this article are not those of the ERC or the European Commission, and the European
Union is not liable for any use that may be made of the information
contained here.}}

\begin{document}

\maketitle

\begin{abstract}
Algorithms and underlying mathematics are presented
for numerical computation with periodic
functions via approximations to machine precision by
trigonometric polynomials, including the solution of
linear and nonlinear periodic ordinary differential equations.
Differences from the nonperiodic Chebyshev case are highlighted.
\end{abstract}



\pagestyle{myheadings}
\thispagestyle{plain}

\markboth{WRIGHT, JAVED, MONTANELLI AND TREFETHEN}
{EXTENSION OF CHEBFUN TO PERIODIC PROBLEMS}

\section{Introduction}
It is well known that trigonometric representations of periodic
functions and Chebyshev polynomial representations of nonperiodic functions
are closely related.  Table~\ref{parallels} lists some of the
parallels between these two situations.  Chebfun, a software
system for computing with functions and solving ordinary
differential equations~\cite{battlestref,chebbook,cacm},
relied entirely on Chebyshev representations in its
first decade.  This paper describes
its extension to periodic problems initiated by
the first author and released with Chebfun Version 5.1
in December 2014.

\begin{table}[h]
\label{parallels}
\caption{Some parallels between trigonometric and Chebyshev
settings.  The row of contributors' names is just a sample
of some key figures.}
\begin{center}
\begin{tabular}{c|c}
{\bf Trigonometric} & {\bf Chebyshev} \\ \hline
\vrule height 14 pt width 0 pt $t\in \pipi$ & $x\in \ones$ \\
periodic & nonperiodic \\
$\exp(ikt)$ & $T_k(x)$ \\
trigonometric polynomials & algebraic polynomials \\
equispaced points & Chebyshev points \\
trapezoidal rule & Clenshaw--Curtis quadrature \\
companion matrix & colleague matrix \\
Horner's rule & Clenshaw recurrence \\
Fast Fourier Transform & Fast Cosine Transform \\
Gauss, Fourier, Zygmund$,\dots$ & Bernstein, Lanczos, Clenshaw$,\dots$ \\
\end{tabular}
\end{center}
\end{table}

Though Chebfun is a software product, the main focus of this
paper is mathematics and algorithms rather than software
{\em per se.} What makes this subject interesting is that
the trigonometric/Chebyshev parallel, though close, is not
an identity.  The experience of building a software system
based first on one kind of representation and then extending
it to the other has given the Chebfun team a uniquely intimate
view of the details of these relationships.  We begin this
paper by listing ten differences between Chebyshev
and trigonometric formulations that we have found important.
This will set the stage for presentations of the problems of
trigonometric series, polynomials, and projections (Section 2),
trigonometric interpolants, aliasing, and barycentric formulas
(Section 3), approximation theory and quadrature (Section 4), and various
aspects of our algorithms (Sections 5--7).

{\em 1. One basis or two.} For working with polynomials on
$\ones$, the only basis functions one needs are the Chebyshev
polynomials $T_k(x)$.  For trigonometric polynomials on $\pipi$,
on the other hand, there are two equally good equivalent
choices: complex exponentials $\exp(ikt)$, or sines and
cosines $\sin(kt)$ and $\cos(kt)$. 
The former is mathematically
simpler; the latter is mathematically more elementary and
provides a framework for dealing with even and odd symmetries.
A fully useful software system for periodic functions needs
to offer both kinds of representation.

{\em 2. Complex coefficients.}
In the $\exp(ikt)$ representation, the expansion coefficients
of a real periodic function are complex.  Mathematically,
they satisfy certain symmetries, and a software system needs to
enforce these symmetries to avoid 
imaginary rounding errors.  Polynomial approximations
of real nonperiodic functions, by contrast, do not lead to
complex coefficients.

{\em 3. Even and odd numbers of parameters.}
A polynomial of degree $n$ is determined by $n+1$ parameters,
a number that may be odd or even.  A trigonometric
polynomial of degree $n$, by contrast, is determined by $2n+1$
parameters, always an odd number, as a consequence of the
$\exp(\pm inx)$ symmetry. For most purposes it is unnatural
to speak of trigonometric polynomials with an even number of
degrees of freedom.  Even numbers make sense, on the other hand,
in the special case of trigonometric polynomials defined by
interpolation at equispaced points, if one imposes the symmetry
condition that the interpolant of the $(-1)^j$ sawtooth should
be real, i.e., a cosine rather than a complex exponential.
Here distinct formulas are needed for the even and odd cases.

{\em 4. The effect of differentiation.}
Differentiation lowers the degree of an algebraic polynomial,
but it does not lower the degree of a trigonometric polynomial; indeed
it enhances the weight of its highest-degree components.

{\em 5. Uniform resolution across the interval.}
Trigonometric representations have uniform properties across
the interval of approximation,
but polynomials are nonuniform, with much greater
resolution power near the ends of $\ones$ than near the
middle~\cite[chap.~22]{atap}.

{\em 6. Periodicity and translation-invariance.}
The periodicity of trigonometric representations means
that a periodic chebfun constructed on $\pipi$, say, can be perfectly
well evaluated at $10\pi$ or $100\pi$; nonperiodic chebfuns have
no such global validity.  Thus, whereas 
interpolation and extrapolation are utterly different for 
polynomials, they are not so different in the trigonometric case.
A subtler consequence of translation
invariance is explained in the footnote on p.~\pageref{footn}.

{\em 7. Operations that break periodicity.}  
A function that is smooth and periodic may lose these properties
when restricted to a subinterval or subjected to 
operations like rounding or absolute value.
This elementary fact has the consequence that
a number of operations on periodic chebfuns require 
their conversion to nonperiodic form.

{\em 8. Good and bad bases.}
The functions $\exp(ikt)$ or $\sin(kt)$ and $\cos(kt)$
are well-behaved by any measure, and nobody would normally
think of using any other basis functions for representing
trigonometric functions.  For polynomials, however,
many people would reach for the basis of monomials $x^k$
before the Chebyshev polynomials $T_k(x)$.  Unfortunately,
the monomials are exponentially ill-conditioned on $\ones$:
a degree-$n$ polynomial of size $1$ on $\ones$ will typically
have coefficients of order $2^n$ when expanded in the basis
$1,x,\dots,x^n$.  Use of this basis will cause trouble in
almost any numerical calculation unless $n$ is very small.

{\em 9. Good and bad interpolation points.}
For interpolation of periodic functions, nobody would normally
think of using any interpolation points other than equispaced.
For interpolation of nonperiodic functions by polynomials,
however, equispaced points lead to exponentially ill-conditioned
interpolation problems~\cite{ptk,runge}.  The mathematically
appropriate choice is not obvious until one learns it: Chebyshev
points, quadratically clustered near $\pm 1$.

{\em 10. Familiarity.}
All the world knows and trusts Fourier analysis.
By contrast, experience with Chebyshev polynomials is often
the domain of experts, and it is not as widely appreciated that
numerical computations based on polynomials can be trusted.
Historically, points 8 and 9 of this list have led to
this mistrust.

The book {\em Approximation Theory and Approximation Practice}~\cite{atap}
summarizes the mathematics and algorithms of Chebyshev technology
for nonperiodic functions.  The present paper was written
with the goal in mind of compiling analogous information
in the trigonometric case.  In particular,
Section 2 corresponds to Chapter 3 of~\cite{atap}, Section 3 to
Chapters 2, 4, and 5, and Section 4 to Chapters 6, 7, 8, 10, and 19.

\section{Trigonometric series, polynomials, and projections}
Throughout this paper, we assume $f$ is
a Lipschitz continuous periodic function on 
$\pipi$.
Here and in all our statements about periodic
functions, the interval $\pipi$ should be understood
periodically: $t=0$ and $t=2\pi$ are identified,
and any smoothness assumptions apply across this point in the
same way as for $t\in (0,2\pi)$~\cite[chap.~1]{katznelson}.

It is known that $f$ has a unique trigonometric series,
absolutely and uniformly convergent, of the form
\begin{equation}
f(t) = \sum_{k=-\infty}^\infty \ck e^{ikt},
\label{series1}
\end{equation}
with Fourier coefficients
\begin{equation}
\ck = {1\over 2\pi} \int_0^{2\pi} f(t) e^{-ikt} dt.
\label{coeffs1}
\end{equation}
(All coefficients in our discussions are
in general complex, though in cases of certain symmetries
they will be purely real or imaginary.)
Equivalently, we have
\begin{equation}
f(t) = \sum_{k=0}^\infty \ak \cos(kt) + \sum_{k=1}^\infty \bk \sin(kt),
\label{series2}
\end{equation}
with $a_0^{} = c_0^{}$ and
\begin{equation}
\ak = {1\over \pi} \int_0^{2\pi} f(t) \cos(kt) dt, \quad
\bk = {1\over \pi} \int_0^{2\pi} f(t) \sin(kt) dt \qquad \rlap{$(k\ge 1).$}
\label{coeffs2}
\end{equation}
The formulas (\ref{coeffs2}) can be derived by matching
the $e^{ikt}$ and $e^{-ikt}$ terms of (\ref{series2}) with
those of (\ref{series1}), which yields the identities 
\begin{equation}
\ck = {a_k^{}\over 2} + {\bk\over 2i},\quad
\cmk = {a_k^{}\over 2} - {\bk\over 2i} \qquad \rlap{$(k\ge 1),$}
\label{abccoeffs1}
\end{equation}
or equivalently,
\begin{equation}
\ak = \ck + \cmk, \quad \bk = i(\ck - \cmk) \qquad\rlap{$(k\ge 1).$}
\label{abccoeffs2}
\end{equation}
Note that if $f$ is real, then (\ref{coeffs2}) implies that
$\ak$ and $\bk$ are real.  The coefficients
$\ck$ are generally complex, and (\ref{abccoeffs1}) implies that they
satisfy $\cmk = \overline{c}_k^{}$.

The {\em degree $n$ trigonometric projection} of $f$ is the function
\begin{equation}
\fn(t) = \sum_{k=-n}^n \ck e^{ikt},
\label{trigpoly1}
\end{equation}
or equivalently
\begin{equation}
\fn(t) = \sum_{k=0}^n \ak \cos(kt) + \sum_{k=1}^n \bk \sin(kt).
\label{trigpoly2}
\end{equation}
More generally, we say that a function of the form
(\ref{trigpoly1})--(\ref{trigpoly2}) is a
{\em trigonometric polynomial of degree $n$}, and
we let $\Pn$ denote the $(2n+1)$-dimensional vector
space of all such polynomials.
The trigonometric projection $\fn$
is the least-squares approximant to $f$ in $\Pn$, i.e., the
unique best approximation to $f$ in the $L^2$ norm over $\pipi$.

\section{Trigonometric interpolants, aliasing, and barycentric formulas}
Mathematically, the simplest degree $n$ trigonometric approximation of
a periodic function $f$ is its trigonometric projection
(\ref{trigpoly1})--(\ref{trigpoly2}).  This approximation
depends on the values of $f(t)$ for all $t\in\pipi$ via
(\ref{coeffs1}) or (\ref{coeffs2}).  Computationally, a simpler
approximation of $f$ is its degree $n$ {\em trigonometric
interpolant}, which only depends on the values at
certain interpolation points.  
In our basic configuration, we wish to interpolate $f$ in equispaced
points by a function $\pn \in \Pn$.
Since the dimension of $\Pn$ is $2n+1$, there should be
$2n+1$ interpolation points.  We take these
{\em trigonometric points} to be
\begin{equation}
\tk = {2\pi k\over N}, \qquad 0\le k \le N-1
\label{trigpts}
\end{equation}
with $N=2n+1$.
The trigonometric interpolation problem goes back at least to 
the young Gauss's calculations of the orbit of the asteroid
Ceres in 1801~\cite{gauss}.

It is known that there exists a unique interpolant
$\pn\in\Pn$ to any set of data values $\fk = f(\tk)$.
Let us write $\pn$ in the form
\begin{equation}
\pn(t) = \sum_{k=-n}^n \ckt e^{ikt},
\label{interp1}
\end{equation}
or equivalently
\begin{equation}
\pn(t) = \sum_{k=0}^n \akt \cos(kt) + \sum_{k=1}^n \bkt \sin(kt),
\label{interp2}
\end{equation}
for some coefficients $\tilde c_{-n}^{},\dots,\tilde c_n^{}$
or equivalently $\tilde a_0^{},\dots,\tilde a_n^{}$
and $\tilde b_1^{},\dots,\tilde b_n^{}$.
The coefficients $\ckt$ and $\ck$ are related by
\begin{equation}
\ckt = \sum_{j=-\infty}^\infty c_{k + jN}^{}
\qquad\rlap{$(|k| \le n)$}
\label{aliasing1}
\end{equation}
(the {\em Poisson summation formula}),
and similarly $\akt$/$\bkt$ and $\ak$/$\bk$ are related by
$\tilde a_0^{} = \sum_{j=0}^\infty a_{jN}^{}$ and
\begin{equation}
\akt = \ak + \sum_{j=1}^\infty (a_{k+jN}^{} + a_{-k+jN}^{}),
\quad \bkt = \bk + \sum_{j=1}^\infty (b_{k+jN}^{} - b_{-k+jN}^{})
\label{aliasing2}
\end{equation}
for $1\le k \le n$.
We can derive these formulas by considering the
phenomenon of {\em aliasing.} For all~$j$, the functions
$\exp(i[k+jN]t)$ take the same 
values at the trigonometric points (\ref{trigpts}).  This implies that
$f$ and the trigonometric polynomial (\ref{interp1}) with
coefficients defined by (\ref{aliasing1}) take the same
values at these points.  In other words, (\ref{interp1}) is the
degree $n$ trigonometric interpolant to $f$.
A similar argument justifies (\ref{interp2})--(\ref{aliasing2}).

Another interpretation of the coefficients $\ckt, \akt, \bkt$ is
that they are equal to the approximations to $\ck, \ak, \bk$
one gets if the integrals (\ref{coeffs1}) and (\ref{coeffs2}) are
approximated by the periodic trapezoidal quadrature rule with $N$
points~\cite{tw}:
\begin{equation}
\ckt = {1\over N} \sum_{j=0}^{N-1} \fj e^{-ik\tj} ,
\label{coeffs3}
\end{equation}
\begin{equation}
\akt = {2\over N} \sum_{j=0}^{N-1} \fj \cos(k\tj), \quad
\bkt = {2\over N} \sum_{j=0}^{N-1} \fj \sin(k\tj) \qquad\rlap{$(k\ge 1).$}
\label{coeffs4}
\end{equation}
To prove this, we note that the trapezoidal rule computes
the same Fourier coefficients for $f$ as for $\pn$, since
they take the same values at the grid points;
but these must be equal to the true Fourier coefficients of $\pn$,
since the $N=(2n+1)$-point trapezoidal rule is exactly correct
for $e^{-2int}, \dots, e^{2int}$, hence for any
trigonometric polynomial of degree $2n$, hence in particular for any
trigonometric polynomial of degree $n$ times an exponential $\exp(-ikt)$
with $|k|\le n$.
From (\ref{coeffs3})--(\ref{coeffs4}) it is evident that
the discrete Fourier coefficients
$\ckt$, $\akt$, $\bkt$ can be computed by
the Fast Fourier Transform (FFT), which, in fact, Gauss invented
for this purpose.

Suppose one wishes to evaluate the interpolant $\pn(t)$ at certain points
$t$.  One good algorithm is to compute 
the discrete Fourier coefficients and then apply them.
Alternatively, another good approach is to
perform interpolation directly by means of the {\em barycentric
formula} for trigonometric interpolation, introduced by
Salzer~\cite{salzer} and later simplified by Henrici~\cite{henrici}:
\begin{equation}
\pn(t) = \sum_{k=0}^{N-1} (-1)^k f_k \csc({t-\tk\over 2})
\left/\, \sum_{k=0}^{N-1} (-1)^k \csc({t-\tk\over 2}) \right. 
\rlap{\quad ($N$  odd).}
\label{bary1}
\end{equation}
(If $t$ happens to be exactly equal to a grid point $\tk$,
one takes $\pn(t) = \fk$.)
The work involved in this formula
is just $O(N)$ operations per evaluation, and stability has
been established (after a small modification) in~\cite{AX}.
In practice, we find the Fourier coefficients and barycentric
formula methods equally effective.

In the above discussion, we have assumed that the number
of interpolation points, $N$, is odd.  However,
trigonometric interpolation, unlike trigonometric projection,
makes sense for an even number of degrees of freedom too
(see e.g.~\cite{faber,kress,zygmund});
it would be surprising if
FFT codes refused to accept input vectors of even lengths!
Suppose $n\ge 1$ is given and we wish to
interpolate $f$ in $N=2n$ trigonometric points (\ref{trigpts})
rather than $N=2n+1$.  This is one data value less than usual for
a trigonometric polynomial of this degree,
and we can lower the number of degrees of
freedom in (\ref{interp1}) by imposing the condition
\begin{equation}
\tilde c_{-n}^{} = \tilde c_n^{}
\label{cond1}
\end{equation}
or equivalently in (\ref{interp2}) by imposing the condition
\begin{equation}
\tilde b_n^{} = 0.
\label{cond2}
\end{equation}
This amounts to prescribing that the trigonometric interpolant
through sawtoothed data of the form $\fk = (-1)^k$
should be $\cos(nt)$ rather than some other function such as
$\exp(int)$---the only choice that ensures that
real data will lead to a real interpolant.  
An equivalent prescription is that an arbitrary number $N$
of data values, even or odd, will be interpolated by a linear
combination of the first $N$ terms of the sequence
\begin{equation}
1,\, \cos(t),\,  \sin(t),\,  \cos(2t),\, \sin(2t),\, \cos(3t),\, \dots.
\label{specialset}
\end{equation}

In this case of trigonometric interpolation with $N$ even, the formulas
(\ref{trigpts})--(\ref{coeffs4}) still hold, except that 
(\ref{aliasing1}) and (\ref{coeffs3}) must be multiplied by $1/2$
for $k = \pm n$.  FFT codes, however, do not
store the information that way.  Instead, following (\ref{cond1}), they
compute $\tilde a_{-n}^{}$ by (\ref{coeffs3}) with $2/N$ instead
of $1/N$ out front---thus effectively storing
$\tilde c_{-n}^{} +\tilde c_n^{}$ in the place of $\tilde c_{-n}^{}$---and
then apply (\ref{interp1}) with the $k=n$ term omitted.
This gives the right result for values of $t$ on
the grid, but not at points in-between.

Note that the conditions (\ref{cond1})--(\ref{specialset}) are
very much tied to
the use of the sample points (\ref{trigpts}).  If the grid
were translated uniformly, then different relationships
between $c_n^{}$ and $c_{-n}^{}$ or $a_n^{}/b_n^{}$ and
$a_{-n}^{}/b_{-n}^{}$ would be appropriate in 
(\ref{cond1})--(\ref{cond2}) and
different basis functions in (\ref{specialset}), and if the grid
were not uniform, then it would be hard to justify any particular
choices at all for even $N$.
For these reasons, even numbers of degrees of freedom make sense
in equispaced interpolation but not in other trigonometric
approximation contexts, in general.
Henrici~\cite{henrici} provides a modification
of the barycentric formula (\ref{bary1}) for the equispaced case $N=2n$.

\section{Approximation theory and quadrature}
The basic question of approximation theory is, will 
approximants to a function $f$ converge as the degree
is increased, and how fast?
The formulas of the last two
sections enable us to derive theorems addressing this question
for trigonometric projection and interpolation.
(For finer points of trigonometric approximation theory,
see~\cite{meinardus}.)
The smoother $f$ is, the faster its Fourier coefficients
decrease, and the faster the convergence of the approximants.
(If $f$ were merely continuous rather than
Lipschitz continuous, then the trigonometric version of the Weierstrass
approximation theorem~\cite[Section I.2]{katznelson}
would ensure that it could be approximated
arbitrarily closely by trigonometric polynomials,
but not necessarily by projection or interpolation.)

Our first theorem asserts that Fourier coefficients
decay algebraically if $f$ has a finite number of derivatives,
and geometrically
if $f$ is analytic.  Here and in Theorem~\ref{thm2} below,
we make use of the notion of the {\em total variation,} $V$, of
a periodic function $\varphi$ defined on $\pipi$, which is defined
in the usual way as the supremum of all sums $\sum_{i=1}^n 
|\varphi(x_i)-\varphi(x_{i-1})|$, where $\{x_i\}$
are ordered points in $\pipi$ with $x_0 = x_n$; $V$ is equal to the
the $1$-norm of $f'$, interpreted if necessary
as a Riemann--Stieltjes integral~\cite[Section I.4]{katznelson}.
Thus $|\sin(t)|$ on $\pipi$, for example, corresponds to $\nu =1$,
and $|\sin(t)|^3$ to $\nu =3$.  
All our theorems continue to assume that $f$ is $2\pi$-periodic.

\begin{theorem}
\label{thm1}
If $f$ is $\nu\ge 0$ times differentiable and
$f^{(\nu)}$ is of bounded variation $V$ on $\pipi$, then
\begin{equation}
|\ck| \le {V\over 2\pi |k|^{\nu+1}}.
\label{est1}
\end{equation}
If $f$ is analytic with $|f(t)|\le M$ in the open strip of
half-width $\alpha$ around the real axis in the complex $t$-plane, then
\begin{equation}
|\ck| \le  M e^{-\alpha|k|} .
\label{est2}
\end{equation}
\end{theorem}

\begin{proof}
The bound (\ref{est1}) can be derived by integrating (\ref{coeffs1}) by
parts $\nu+1$ times.  
Equation (\ref{est2}) can be derived by shifting the interval
of integration $\pipi$ of (\ref{coeffs1}) downward in
the complex plane for $k>0$, or upward for $k<0$, by a
distance arbitrarily close to $\alpha$; see~\cite[Section 3]{tw}.
\end{proof}

To apply Theorem~\ref{thm1} to trigonometric approximations, we
note that the error in the degree $n$ trigonometric
projection (\ref{trigpoly1}) is
\begin{equation}
f(t) - \fn(t) = \sum_{|k|>n} \ck e^{ikt} ,
\label{error1}
\end{equation}
a series that converges absolutely and uniformly by the Lipschitz
continuity assumption on $f$.  Similarly, (\ref{aliasing1}) implies
that the error in trigonometric interpolation is
\begin{equation}
f(t) - \pn(t) = \sum_{|k|>n} \ck (e^{ikt} - e^{ik'\kern -1pt t}), 
\label{error2}
\end{equation}
where $k' = \hbox{mod}(k+n,2n+1)-n$ is the index that $k$ gets
aliased to on the $(2n+1)$-point grid, i.e., the integer of absolute value
$\le n$ congruent to $k$ modulo $2n+1$.
These formulas give us bounds on the error in trigonometric
projection and interpolation.

\begin{theorem}
\label{thm2}
If $f$ is $\nu\ge 1$ times differentiable and
$f^{(\nu)}$ is of bounded variation $V$ on $\pipi$, then
its degree $n$ trigonometric projection and interpolant satisfy 
\begin{equation}
\|f - \fn\|_\infty^{} \le {V\over \pi\kern .7pt \nu\kern .7pt  n^\nu}, \qquad
\|f - \pn\|_\infty^{} \le {2V\over \pi\kern .7pt \nu \kern .7pt n^\nu}.
\label{est3}
\end{equation}
If $f$ is analytic with $|f(t)|\le M$ in the open strip of
half-width $\alpha$ around the real axis in the complex $t$-plane, they
satisfy
\begin{equation}
\|f-\fn\|_\infty^{} \le {2M e^{-\alpha n}\over e^\alpha-1}, \qquad
\|f-\pn\|_\infty^{} \le {4M e^{-\alpha n}\over e^\alpha-1} .
\label{est4}
\end{equation}
\end{theorem}

\begin{proof}
The estimates (\ref{est3}) follow by bounding the 
tails (\ref{error1}) and (\ref{error2}) with (\ref{est1}), and
(\ref{est4}) likewise by bounding them with (\ref{est2}).
\end{proof}

A slight variant of this argument gives an estimate for quadrature.
If $I$ denotes the integral of a function $f$ over $\pipi$ 
and $\IN$ its approximation by the $N$-point
periodic trapezoidal rule, then from (\ref{coeffs1}) and (\ref{coeffs3}),
we have $I = 2\pi c_0^{}$ and $\IN = 2\pi \tilde c_0^{}$.
By (\ref{aliasing1}) this implies
\begin{equation}
\IN - I = 2\pi \sum_{j\ne 0} c_{jN}^{},
\label{trapest}
\end{equation}
which gives the following result.
\begin{theorem}
\label{thm4}
If $f$ is $\nu\ge 1$ times differentiable and
$f^{(\nu)}$ is of bounded variation $V$ on $\pipi$, then
the $N$-point periodic trapezoidal rule
approximation to its integral over $\pipi$ satisfies
\begin{equation}
|\IN - I| \le {4 V\over N^{\nu+1}}.
\label{trap1}
\end{equation}
If $f$ is analytic with $|f(t)|\le M$ in the open strip of
half-width $\alpha$ around the real axis in the complex $t$-plane,
it satisfies
\begin{equation}
|\IN-I| \le {4\pi M \over e^{\alpha N}-1}.
\label{trap2}
\end{equation}
\end{theorem}

\begin{proof}
These results follow by bounding (\ref{trapest}) with 
(\ref{est1}) and (\ref{est2}) as in the proof of Theorem~\ref{thm2}.
From (\ref{est1}), the bound one gets
is $2V\zeta(\nu+1)/N^{\nu+1}$, where $\zeta$ is the
Riemann zeta function, which we have simplified by
the inequality $\zeta(\nu+1)\le \zeta(2) < 2$ for $\nu\ge 1$.
The estimate (\ref{trap2}) originates with Davis~\cite{davis};
see also~\cite{kress,tw}.
\end{proof}

Finally, in a section labeled ``Approximation theory'' we must mention another
famous candidate for periodic function
approximation: best approximation in the $\infty$-norm. 
Here the trigonometric version of the Chebyshev alternation theorem holds,
assuming $f$ is real.
This result is illustrated below in Figure~\ref{bestapprox}.

\begin{theorem}
Let $f$ be real and continuous on
the periodic interval $\pipi$.  For each $n\ge 0$,
$f$ has a unique best approximant $\pns\in \Pn$ with
respect to the norm $\|\cdot\|_\infty^{}$, and $\pns$ is characterized
by the property that the error curve $(f-\pns)(t)$ equioscillates on
$[\kern .5pt 0,2\pi)$ between at least $2n+2$ equal extrema $\pm\|f-\pns\|_\infty^{}$ of
alternating signs.
\end{theorem}
\begin{proof}
See~\cite[Section 5.2]{meinardus}.
\end{proof}

\section{Trigfun computations}
Building on the mathematics of the past three sections,
Chebfun was extended in 2014 to incorporate trigonometric
representations of periodic functions alongside its traditional
Chebyshev representations.
(Here and in
the remainder of the paper, we assume the reader is familiar
with Chebfun.)
Our convention is that a {\em trigfun} is a
representation via coefficients $\ck$ as in (\ref{trigpoly1})
of a sufficiently smooth
periodic function $f$ on an interval by a trigonometric
polynomial of adaptively determined degree, the aim always being
accuracy of 15 or 16 digits relative
to the $\infty$-norm of the function on the interval.
This follows the same pattern as traditional Chebyshev-based
chebfuns, which are representations of nonperiodic functions
by polynomials, and a trigfun is not a distinct object from a
chebfun but a particular type of chebfun.  The default interval,
as with ordinary chebfuns, is $\ones$, and other intervals are
handled by the obvious linear
transplantation.\footnote{\label{footn}Actually,
one aspect of the transplantation is not obvious, an indirect
consequence of the translation-invariance
of trigonometric functions.
The nonperiodic function $f(x) = x$ defined on $[-1,1]$, for
example, has Chebyshev coefficients $a_0^{}=0$ and $a_1^{} = 1$,
corresponding to the expansion $f(x) = 0T_0^{}(x) + 1T_1^{}(x)$.
Any user will expect the transplanted function $g(x) = x-1$
defined on $[0,2]$ to have the same coefficients $a_0^{}=0$
and $a_1^{} = 1$, corresponding to the transplanted expansion
$g(x) = 0T_0^{}(x-1) + 1T_1^{}(x-1)$, and this is what Chebfun
delivers.  By contrast, consider the periodic function $f(t)
= \cos t$ defined on $[-\pi,\pi]$ and its transplant $g(t) =
\cos(t-\pi) = -\cos t$ on $\pipi$.  A user will expect the
expansion coefficients of $g$ to be not the same as those
of $f$, but their negatives!  This is because we expect to
use the same basis functions $\exp(ikx)$ or $\cos(kx)$ and
$\sin(kx)$ on any interval of length $2\pi$, however translated.
The trigonometric part of Chebfun is designed accordingly.}

For example, here we construct and plot a trigfun for $\cos(t) + 
\sin(3t)/2$ on $\pipi$:

\begin{figure}
\begin{center}
\includegraphics[scale=.5]{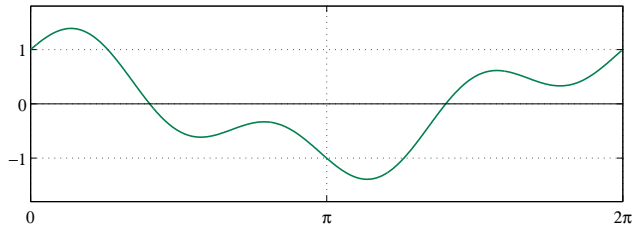}
\end{center}
\caption{\label{fig1a} The trigfun representing
$f(t) = \cos(t)+ \sin(3t)/2$ on $\pipi$.  One can evaluate
$f$ with\/ {\tt f(t)}, compute its definite integral with\/ {\tt sum(f)}
or its maximum with\/ {\tt max(f)},
find its roots with {\tt roots(f)}, and so on.
}
\end{figure}

{\small
\topspace\begin{verbatim}
        >> f = chebfun('cos(t) + sin(3*t)/2', [0 2*pi], 'trig'), plot(f)
\end{verbatim}
\bottomspace\par}

\noindent
The plot appears in Figure~\ref{fig1a}, and
the following text output is produced, with the flag {\tt trig} signalling
the periodic representation.

{\small
\topspace\begin{verbatim}
        f =
           chebfun column (1 smooth piece)
               interval       length     endpoint values trig
        [       0,     6.3]        7         1        1 
\end{verbatim}
\bottomspace\par}

\noindent
We see that Chebfun has determined that this function
$f$ is of length $N=7$.  This means
that there are $7$ degrees of freedom, i.e., $f$ is a trigonometric polynomial
of degree $n=3$, whose coefficients we can extract with
{\tt c = trigcoeffs(f)}, or in cosine/sine form
with {\tt [a,b] = trigcoeffs(f)}.

Note that the Chebfun constructor does not analyze its input
symbolically, but just evaluates the function at trigonometric points
(\ref{trigpts}), and from this information the degree
and the values of the coefficients are determined.
The constructor also detects when a function is real.
A trigfun constructed in the ordinary manner
is always of odd length $N$, corresponding to a trigonometric
polynomial of degree $n = (N-1)/2$, though it is
possible to make even-length trigfuns by explicitly specifying $N$.

To construct a trigfun, Chebfun samples the function on grids of
size $16, 32, 64,\dots$ and tests the resulting discrete Fourier
coefficients for convergence down to relative machine
precision.  (Powers of 2 are used
since these are particularly efficient for the FFT, even though
the result will ultimately be trimmed to an odd number
of points.  As with non-trigonometric Chebfun,
the engineering details are complicated and under ongoing
development.)  When convergence is achieved, the
series is chopped at an appropriate point and the degree reduced
accordingly.

Once a trigfun has been created, computations can be carried out
in the usual Chebfun fashion via overloads of familiar
MATLAB commands.  For example,

{\small
\topspace\begin{verbatim}
        >> sum(f.^2)
        ans = 3.926990816987241
\end{verbatim}
\bottomspace\par}

\noindent
This number is computed by integrating the trigonometric
representation of $f^2$, i.e., by returning the number $2\pi c_0^{}$
corresponding to the trapezoidal rule applied to $f^2$ as
described around Theorem~\ref{thm4}.
The default 2-norm is the square root of this result,

{\small
\topspace\begin{verbatim}
        >> norm(f)
        ans = 1.981663648803005
\end{verbatim}
\bottomspace\par}

\noindent
Derivatives of functions are computed
by the overloaded command {\tt diff}.
(In the unusual case where a trigfun has been constructed of
even length, differentiation will increase its length by $1$.)
The zeros of $f$ are found with {\tt roots}:

{\small
\topspace\begin{verbatim}
        >> roots(f)
        ans =
           1.263651122898791
           4.405243776488583
\end{verbatim}
\bottomspace\par}

\noindent
and Chebfun determines maxima and minima by first computing
the derivative, then checking all of its roots:

{\small
\topspace\begin{verbatim}
        >> max(f)
        ans = 1.389383416980387
\end{verbatim}
\bottomspace\par}

\noindent
Concerning the algorithm used for periodic rootfinding,
one approach would be to solve a companion
matrix eigenvalue problem, and $O(n^2)$ algorithms for
this task have recently been developed~\cite{amvw}.
When development of these methods settles down, they may be incorporated
in Chebfun.  For the moment, trigfun rootfinding
is done by first converting the problem to nonperiodic Chebfun form using
the standard Chebfun constructor,
whereupon we take advantage of Chebfun's $O(n^2)$ recursive interval
subdivision strategy~\cite{bt}.  
This shifting to subintervals for rootfinding is an
example of an operation that breaks periodicity as mentioned in
item 7 of the introduction.

The main purpose of the periodic part of Chebfun is to
enable machine precision computation with periodic
functions that are not exactly trigonometric polynomials.
For example, 
$\exp(\sin t)$ on $\pipi$ is represented by a trigfun
of length $27$, i.e., a
trigonometric polynomial of degree 13:

{\small
\topspace\begin{verbatim}
        g = chebfun('exp(sin(t))', [0 2*pi], 'trig')
        g =
           chebfun column (1 smooth piece)
               interval       length     endpoint values trig
        [       0,     6.3]       27         1        1 
\end{verbatim}
\bottomspace\par}

\noindent
The coefficients can be plotted on a log scale with the
command {\tt plotcoeffs(f)},
and the in 
Figure~\ref{fig2} reveals the faster-than-geometric decay
of an entire function.

\begin{figure}
\begin{center}
\includegraphics[scale=.5]{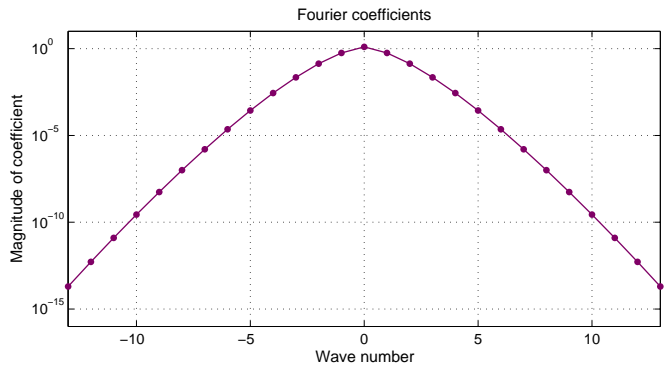}
\end{center}
\caption{\label{fig2} Absolute values of the Fourier coefficients of
the trigfun for\/ $\exp(\sin t)$ on $\pipi$.
This is an entire function (analytic throughout
the complex $t$-plane), and in accordance with Theorem~$\ref{thm1}$, the
coefficients decrease faster than geometrically.}
\end{figure}

Figure~\ref{twomore} shows trigfuns and coefficient plots
for $f(t)=\tanh(5\cos(5t))$ and
$g(t)=\exp(-1/\max\{0, 1-t^2/4\})$ on
$[-\pi, \pi]$.  The latter is $C^\infty$ but not analytic.
Figure~\ref{signals} shows a further pair of examples that
we call an ``AM signal'' and an ``FM signal''.  These are
among the preloaded functions available with
{\tt cheb.gallerytrig}, Chebfun's trigonometric analogue
of the MATLAB {\tt gallery} command.

\begin{figure}
\begin{center}
\includegraphics[scale=.6]{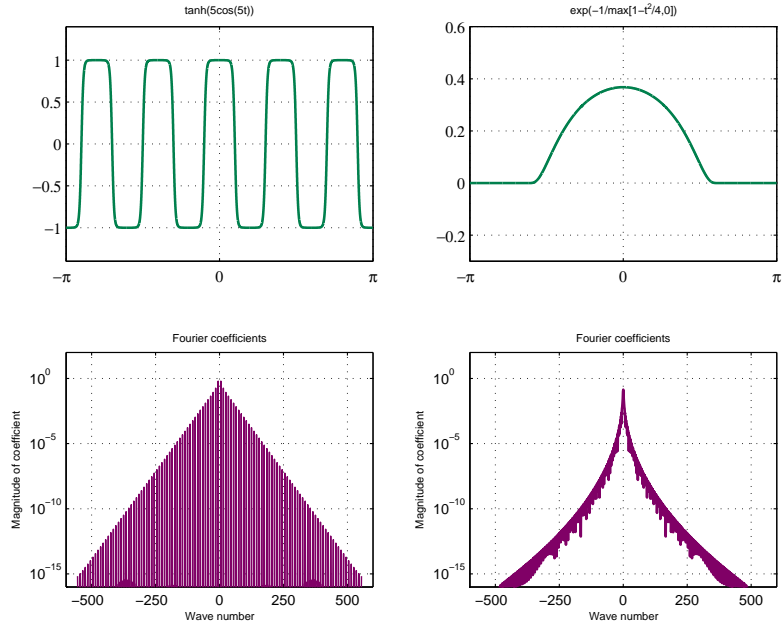}
\end{center}
\caption{\label{twomore} Trigfuns of\/
$\tanh(5\sin t)$ and\/ $\exp(-100(t+.3)^2)$
(upper row) and corresponding absolute values
of Fourier coefficients (lower row).}
\end{figure}

\begin{figure}
\begin{center}
\includegraphics[scale=.6]{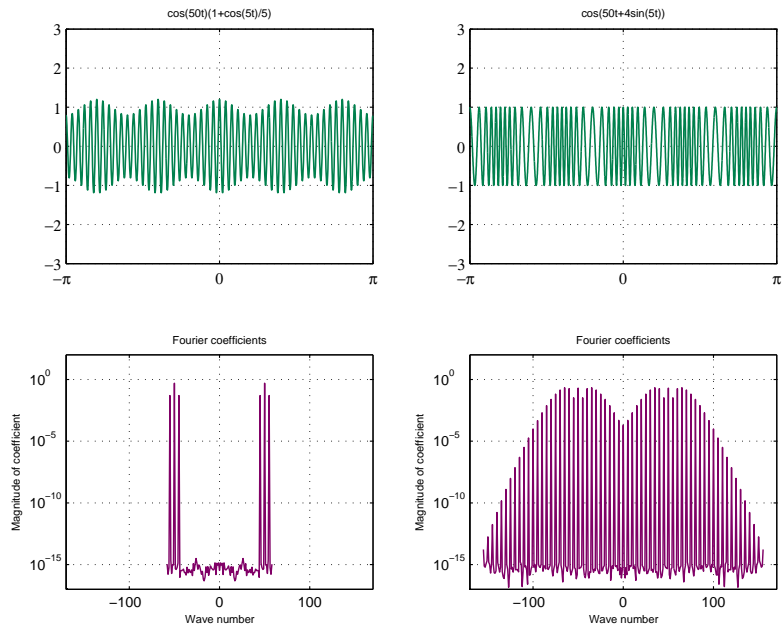}
\end{center}
\caption{\label{signals} Trigfuns of the ``AM signal''
$\cos(50t)(1+\cos(5t)/5)$ and the ``FM signal''
$\cos(50t+4\sin(5t))$ (upper row) and corresponding absolute values
of Fourier coefficients (lower row).}
\end{figure}

Computation with trigfuns, as with nonperiodic chebfuns,
is carried out by
a continuous analogue of floating point arithmetic~\cite{cacm}.
To illustrate the ``rounding'' process involved,
the degrees of the trigfuns above are
555 and 509, respectively.  Mathematically,
their product is of degree 1064.  Numerically, however,
Chebfun achieves 16-digit accuracy with degree 556.

Here is a more complicated example of Chebfun rounding adapted
from~\cite{cacm}, where it is computed with nonperiodic 
representations.

{\small
\topspace\begin{verbatim}
        f = chebfun(@(t) sin(pi*t), 'trig')
        s = f
        for j = 1:15
          f = (3/4)*(1 - 2*f.^4), s = s + f
        end
\end{verbatim}
\bottomspace\par}

\noindent
This program takes 15 steps of an iteration that in principle
quadruples the degree at each step, giving a function
$s$ at the end of degree
$4^{15} = \hbox{1,073,741,824}$.  In actuality,
however, because of the rounding to 16 digits,
the degree comes out one million times smaller as 1148.  This function
is plotted in Figure~\ref{logistic}.  Following~\cite{cacm},
we can compute the roots of $s-8$
in half a second on a desktop machine:

{\small
\topspace\begin{verbatim}
        >> roots(s-8)
        ans =
          -0.992932107411876
          -0.816249934290177
          -0.798886729723433
          -0.201113270276572
          -0.183750065709828
          -0.007067892588112
           0.346696120418255
           0.401617073482111
           0.442269489632475
           0.557730510367530
           0.598382926517899
           0.653303879581760
\end{verbatim}
\bottomspace\par}

\noindent
The integral with {\tt sum(s)} is
$15.265483825826763$, correct except in the last two digits.

\begin{figure}
\begin{center}
\includegraphics[scale=.5]{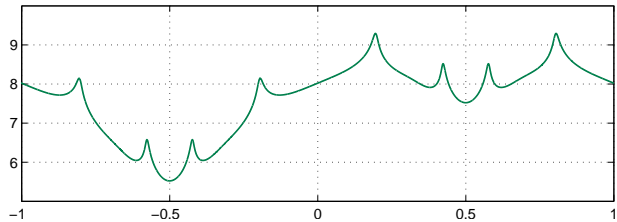}
\end{center}
\caption{\label{logistic} After fifteen steps of an iteration,
this periodic function has degree $1148$ in its Chebfun representation
rather than the mathematically exact figure {\rm 1,073,741,824}.}
\end{figure}

If one tries to construct a trigfun by sampling
a function that is not smoothly periodic,
Chebfun will by default go up to length $2^{16}$ and then
issue a warning:

{\small
\topspace\begin{verbatim}
        >> h = chebfun('exp(t)', [0 2*pi], 'trig')
        Warning: Function not resolved using 65536 pts.
                 Have you tried a non-trig representation? 
\end{verbatim}
\bottomspace\par}

\noindent
On the other hand, computations that are known
to break periodicity or smoothness will result in the representation
being automatically cast from a trigfun to a chebfun.
For example, here we define $g$ to be the
absolute value of the function 
$f(t) = \cos(t) + \sin(3t)/2$ of Figure~\ref{fig1a}.
The system detects that $f$ has zeros,
implying that $g$ will probably not be smooth, and
accordingly constructs it not as a trigfun but as an ordinary chebfun with
several pieces:

{\small
\topspace\begin{verbatim}
        >> f = chebfun('cos(t) + sin(3*t)/2', [0 2*pi], 'trig'), g = abs(f)
        g =
           chebfun column (3 smooth pieces)
               interval       length     endpoint values  
        [       0,     1.3]       17         1  3.8e-16 
        [     1.3,     4.4]       25   1.8e-15  7.3e-16 
        [     4.4,     6.3]       20  -6.6e-18        1 
        Total length = 62.
\end{verbatim}
\bottomspace\par}

\noindent
Similarly, if you add or multiply a trigfun and a chebfun,
the result is a chebfun.

\begin{figure}
\begin{center}
\includegraphics[scale=.5]{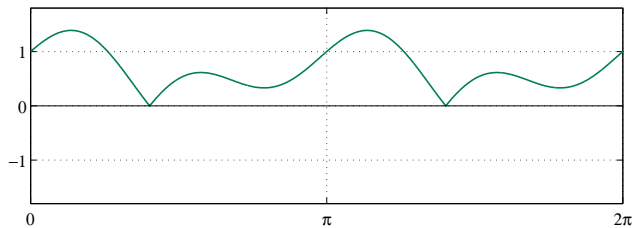}
\end{center}
\caption{\label{fig1b} When the absolute value of
the trigfun $f$ of Figure~$\ref{fig1a}$ is computed,
the result is a nonperiodic chebfun with three smooth pieces.}
\end{figure}

\section{Applications}
Analysis of periodic functions and signals is one of the oldest
topics of mathematics and engineering.  Here we give
six examples of how a system for automating such computations may be useful.

{\em Complex contour integrals.}
Smooth periodic integrals arise ubiquitously in complex analysis.
For example, suppose we wish to determine the number of zeros of
$f(z) = \cos(z) - z$ in the
complex unit disk.  The answer is given by
\begin{equation}
m = {1\over 2\pi i} \int {f'(z)\over f(z) } dz
= {1\over 2\pi i} \int {1\over f(z) } {df\over dt} dt
\label{contourint}
\end{equation}
if $z = \exp(it)$ with $t\in \pipi$.
With periodic Chebfun, we can compute $m$ by

{\small
\topspace\begin{verbatim}
        >> z = chebfun('exp(1i*t)', [0 2*pi], 'trig');
        >> f = cos(z) - z;
        >> m = real(sum(diff(f)./f)/(2i*pi))
        m = 1.000000000000000
\end{verbatim}
\bottomspace\par}

\noindent
Changing the integrand from $f'(z)/f(z)$ to $z f'(z)/f(z)$ 
gives the location of the zero, correct to all digits displayed.

{\small
\topspace\begin{verbatim}
        >> z0 = real(sum(z.*diff(f)./f)/(2i*pi))
        z0 = 0.739085133215161
\end{verbatim}
\bottomspace\par}

\noindent
(The {\tt real} commands are included to remove imaginary
rounding errors.)
For wide-ranging extensions of calculations like these, including
applications to matrix eigenvalue problems, see~\cite{akt}.

{\em Linear algebra.}
Chebfun does not work from explicit formulas: to construct a
function, it is only necessary to be able to evaluate it.
This is an extremely useful feature for linear algebra calculations.
For example, the matrix
\begin{equation}
\def\r#1{\phantom{xx}\llap{$#1$}}
\def\rr#1{\phantom{xx,}\llap{$#1$}}
A = {1\over 3}
\pmatrix{
\r{2} & \rr{-2i} & \r{1} & \r{1} \cr
\r{2i} & \rr{-2} & \r{0} & \r{2} \cr
\r{-2} & \rr{0} & \r{1} & \r{2} \cr
\r{0} & \rr{i} & \r{0} & \r{2}
}
\end{equation}
has all its eigenvalues in the unit disk.  A question with
the flavor of control and stability theory is, what is the
maximum resolvent norm $\|(zI-A)^{-1}\|$ for $z$ on the unit
circle?  We can calculate the answer with the code below, which
constructs a periodic chebfun of degree $n=569$.   The maximum is
$27.68851$, attained with $z = \exp(0.454596\kern.5pt i)$.

{\small
\topspace\begin{verbatim}
        A = [2 -2i 1 1; 2i -2 0 2; -2 0 1 2; 0 1i 0 2]/3, I = eye(4)
        ff = @(t) 1/min(svd(exp(1i*t)*I-A))
        f = chebfun(ff, [0 2*pi], 'trig', 'vectorize')
        [maxval,maxpos] = max(f)
\end{verbatim}
\bottomspace\par}

\begin{figure}
\begin{center}
\includegraphics[scale=.5]{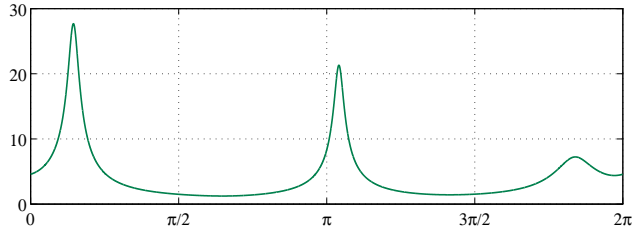}
\end{center}
\caption{\label{resolve} Resolvent norm $\|(zI-A)^{-1}\|$
for a $4\times 4$ matrix $A$ with $z= e^{it}$ on
the unit circle.}
\end{figure}

{\em Circular convolution and smoothing.}
The circular or periodic convolution of two functions
$f$ and $g$ with period $T$ is defined by
\begin{equation}
(f*g)(t) := \int_{t_0}^{t_0 + T} g(s)f(t-s)ds,
\end{equation}
where $t_0$ is aribtrary. Circular convolutions can be
computed for trigfuns with the {\tt circconv} function, whose algorithm
consists of coefficientwise multiplication in Fourier space.
For example, here is a trigonometric interpolant through 
$201$ samples of a smooth function plus noise, shown
in the upper-left panel of Figure~\ref{noisy}.

{\small
\topspace\begin{verbatim}
        N = 201, tt = trigpts(N, [-pi pi]) 
        ff = exp(sin(tt)) + 0.05*randn(N,1) 
        f = chebfun(ff, [-pi pi], 'trig')
\end{verbatim}
\bottomspace\par}

\noindent
The high wave numbers can be smoothed by convolving
$f$ with a mollifier.  Here we use
a Gaussian of standard deviation $\sigma=0.1$ (numerically
periodic for $\sigma\le 0.35$).  The result is
shown in the upper-right panel of the figure.

{\small
\topspace\begin{verbatim}
        gaussian = @(t,sigma) 1/(sigma*sqrt(2*pi))*exp(-0.5*(t/sigma).^2) 
        g = @(sigma) chebfun(@(t) gaussian(t,sigma), [-pi pi], 'trig') 
        h = circconv(f, g(0.1)) 
\end{verbatim}
\bottomspace\par}

\begin{figure}
\begin{center}
\includegraphics[scale=.6]{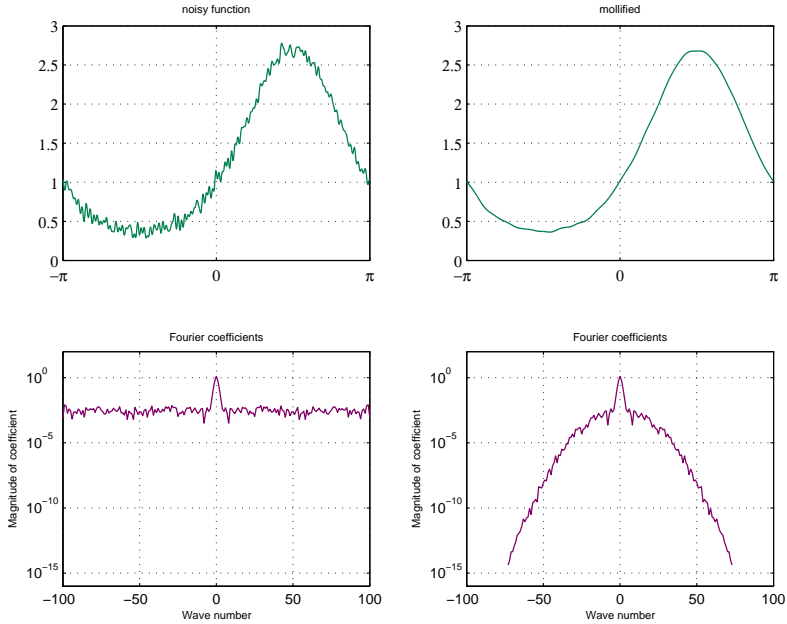}
\end{center}
\caption{\label{noisy} Circular convolution of a noisy function with
a smooth mollifier.}
\end{figure}

{\em Fourier coefficients of non-smooth functions.}
A function $f$ that is not smoothly periodic will at best
have a very slowly converging trigonometric series, but still,
one may be interested in its Fourier coefficients.  These can
be computed by applying {\tt trigcoeffs} to a chebfun representation
of $f$ and specifying how many coefficients are required; the
integrals (\ref{coeffs1}) are then evaluated numerically by
Chebfun's standard method of Clenshaw--Curtis quadrature.
For example,
Figure~\ref{runge} shows a portrayal of the Gibbs phenomenon
from Runge's 1904 book together with its Chebfun equivalent computed
in a few seconds with the commands

{\small
\topspace\begin{verbatim}
        t = chebfun('t', [-pi pi]), f = (abs(t) < pi/2) 
        for N = 2*[1 3 5 7 21 79] + 1
          c = trigcoeffs(f, N) 
          fN = chebfun(c, [-pi pi], 'coeffs', 'trig') 
          plot(fN, 'interval', [0 4*pi])
        end
\end{verbatim}
\bottomspace\par}

\begin{figure}
\begin{center}
\includegraphics[scale=.58]{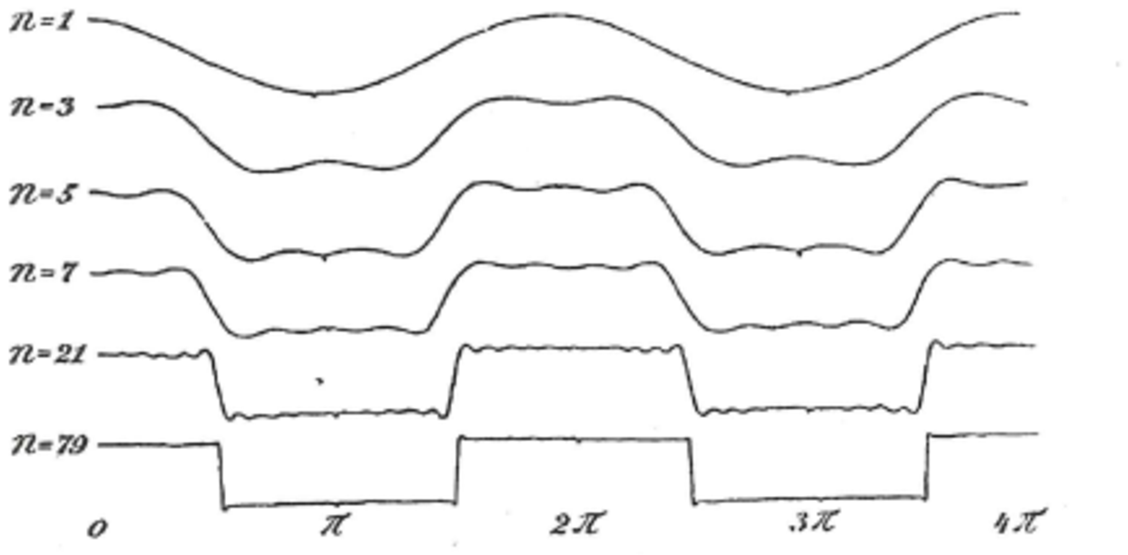}
\includegraphics[scale=.327]{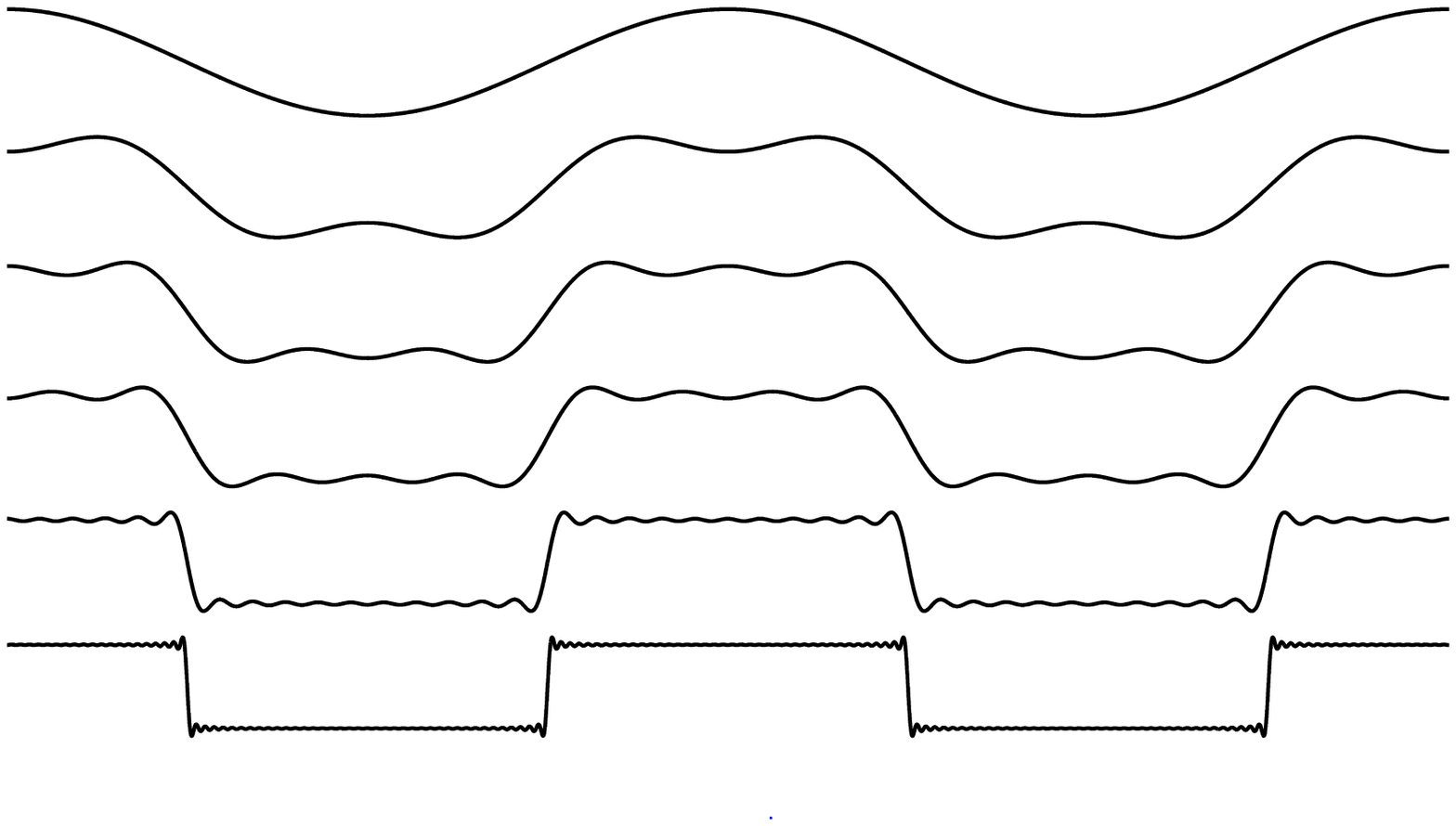}~~
\end{center}
\caption{\label{runge} On the left, a figure from
Runge's\/ $1904$ book\/
{\em Theorie und Praxis der Reihen\/}~{\rm\cite{rungebook}}.
On the right, the
equivalent computed with periodic Chebfun.  Among other things, this
figure illustrates that a trigfun can be accurately evaluated outside
its interval of definition.}
\end{figure}

{\em Interpolation in unequally spaced points.}
Very little attention has been given to trigonometric
interpolation in unequally spaced points, but the 
barycentric formula (\ref{bary1}) for odd $N$ and Henrici's
generalization for even $N$ have been generalized
to this case by Salzer and Berrut~\cite{berrut}.  Chebfun
makes these formulas available through the 
command {\tt chebfun.interp1}, just as has long been
true for interpolation by algebraic polynomials.  For
example, the code

{\small
\topspace\begin{verbatim}
        t = [-3 -2 -1 0 .5 1 1.5 2 2.5]
        p = chebfun.interp1(t, abs(t), 'trig', [-pi pi])
\end{verbatim}
\bottomspace\par}

\noindent
interpolates the function $|t|$ on $[-\pi,\pi]$
in the 9 points indicated by a trigonometric
polynomial of degree $n=4$.
The interpolant is shown 
in Figure~\ref{interpdemo} together with the analogous
curve for equispaced points.

\begin{figure}
\begin{center}
\includegraphics[scale=.6]{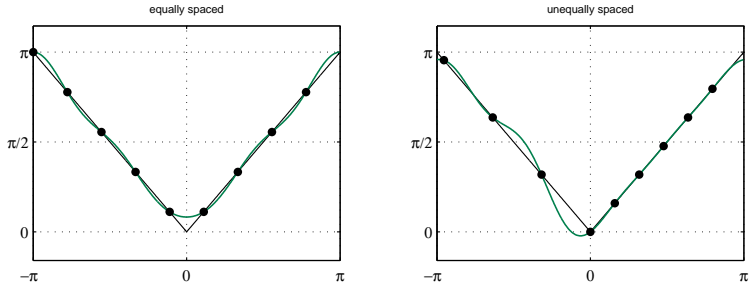}
\end{center}
\caption{\label{interpdemo} 
Trigonometric interpolation of $|t|$ in unequally spaced
points with
the generalized barycentric formula implemented in {\tt chebfun/interp1}.}
\end{figure}

{\em Best approximation, CF approximation, and rational functions.}
Chebfun has long had a dual role: it is a tool
for computing with functions, and also a tool for exploring
principles of approximation theory, including
advanced ones.  The trigonometric side of
Chebfun extends this second aspect to periodic problems.  For example, 
Chebfun's new 
{\tt trigremez} command can compute best approximants
with equioscillating error curves as described in
Theorem~\ref{thm4}~\cite{javed}.
Here is an example that generates the error curve displayed
in Figure~\ref{bestapprox}, with error $12.1095909$.

{\small
\topspace\begin{verbatim}
        f = chebfun('1./(1.01-sin(t-2))', [0 2*pi], 'trig')
        p = trigremez(f,10), plot(f-p)
\end{verbatim}
\bottomspace\par}

\noindent
Chebfun is also acquiring other capabilities for trigonometric
polynomial and rational approximation, including
Carath\'eodory--Fej\'er (CF) near-best approximation via singular
values of Hankel matrices, and these will be described elsewhere.

\begin{figure}
\begin{center}
\includegraphics[scale=.5]{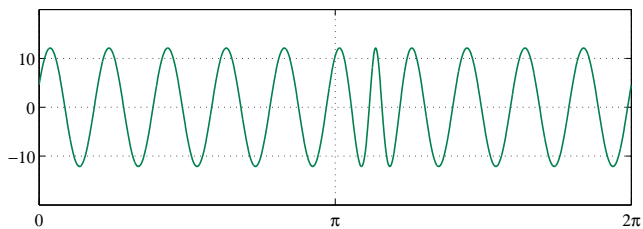}
\end{center}
\caption{\label{bestapprox} Error curve in degree\/ $n=10$ best
trigonometric approximation to
$f(t) = 1/(1.01-\sin(t-2))$ over $\pipi$.  The curve equioscillates between
$2n+2 = 22$ alternating extrema.}
\end{figure}

\section{Periodic ODEs, operator exponentials, and eigenvalue problems}
A major capability of Chebfun is
the solution of linear and nonlinear ordinary differential
equations (ODEs), as well as integral equations,
by applying the backslash command to a ``chebop'' object.
We have extended these capabilities to periodic problems, both
scalars and systems.
See~\cite{eastham}
for the theory of existence and uniqueness
of solutions to periodic ODEs, which goes back to Floquet in the 1880s,
a key point being the avoidance of nongeneric
configurations corresponding to eigenmodes.

Chebfun's algorithm for linear ODEs amounts to
an automatic spectral collocation method wrapped up so that
the user need not be aware of the discretization.  With
standard Chebfun, these are Chebyshev spectral methods, and now
with the periodic extension, they are Fourier
spectral methods~\cite{boydspec}.
The problem is solved on grids of size 32, 64, and so on until
the system judges that the Fourier coefficients have converged
down to the level of noise, and the series is then truncated
at an appropriate point.

For example, consider the problem
\begin{equation}
0.001(u'' + u') - \cos(t) u = 1, \qquad  0\le t \le 6\pi
\label{odelin}
\end{equation}
with periodic boundary conditions.
The following Chebfun code produces the solution
plotted in Figure~\ref{figodelin} in half a second
on a laptop.  Note that the trigonometric discretizations are
invoked by the flag \verb|L.bc = 'periodic'|.

{\small
\topspace\begin{verbatim}
        L = chebop(0,6*pi) 
        L.op = @(x,u) 0.001*diff(u,2) + 0.001*diff(u) - cos(x).*u 
        L.bc = 'periodic' 
        u = L\1 
\end{verbatim}
\bottomspace\par}

\begin{figure}
\begin{center}
\includegraphics[scale=.55]{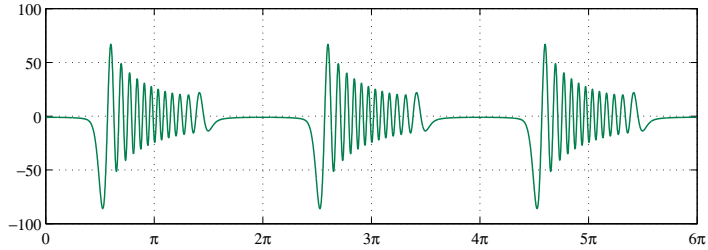}
\end{center}
\caption{\label{figodelin} Solution of the linear periodic ODE\/
$(\ref{odelin})$ as a trigfun of degree\/ $168$, computed by
an automatic Fourier spectral method.}
\end{figure}

\noindent
This trigfun is of degree 168, 
and the residual reported by {\tt norm(L*u-1)} is 
$1\times 10^{-12}$.
As always, $u$ is a chebfun; its maximum, for example, is
$\hbox{\tt max(u)} =66.928$.

For periodic nonlinear ODEs, Chebfun applies trigonometric
analogues of the algorithms developed by Driscoll and Birkisson
in the Chebshev case~\cite{bd1,bd2}.  The whole solution is carried out
by a Newton or damped Newton iteration formulated
in a continuous mode (``solve then discretize'' rather than
``discretize then solve''), with Jacobian matrices replaced
by Fr\'echet derivative operators implemented by means of automatic
differentiation and automatic spectral discretization.
For example, suppose we seek a solution of the nonlinear problem
\begin{equation}
0.004u'' + uu' - u = \cos(2\pi t), \qquad t\in [-1,1]
\label{nonlinprob}
\end{equation}
with periodic boundary conditions.  After seven Newton
steps, the Chebfun commands below
produce the result shown in Figure~\ref{fignonlin}, of
degree $n = 362$, and
the residual norm {\tt norm(N(u)-rhs,'inf')} is reported as
$8\times 10^{-9}$.

{\small
\topspace\begin{verbatim}
        N = chebop(-1,1)
        N.op = @(x,u) .004*diff(u,2) + u.*diff(u) - u
        N.bc = 'periodic'
        rhs = chebfun('cos(2*pi*t)', 'trig')
        u = N\rhs
\end{verbatim}
\bottomspace\par}

\begin{figure}
\begin{center}
\includegraphics[scale=.55]{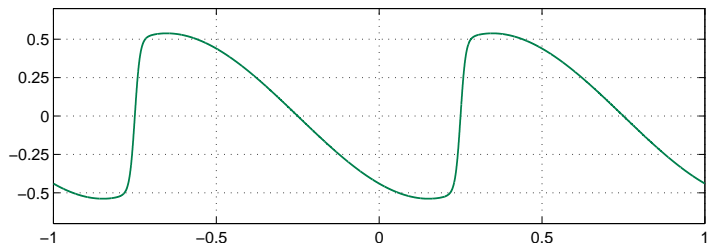}
\end{center}
\caption{\label{fignonlin} Solution of the nonlinear periodic ODE\/
$(\ref{nonlinprob})$ computed by iterating the Fourier
spectral method within a continuous form of
Newton iteration.  Executing\/
{\tt max(diff(u))} shows that the maximum of $u'$ is 
$32.094$.}
\end{figure}

Chebfun's overload of the MATLAB {\tt eigs} command solves 
linear ODE eigenvalue problems by, once again, automated spectral
collocation discretizations~\cite{driscoll}.  This too has
been extended to periodic problems, with Fourier discretizations
replacing Chebyshev.  For example, a famous periodic eigenvalue
problem is the Mathieu equation
\begin{equation}
-u'' + 2 q \cos(2t) u = \lambda u, \qquad t\in \pipi,
\label{mathieueq}
\end{equation}
where $q$ is a parameter.  The commands below give
the plot shown in Figure~\ref{mathieu}.

{\small
\topspace\begin{verbatim}
        q = 2
        L = chebop(@(x,u) -diff(u,2)+2*q*cos(2*x).*u, [0 2*pi], 'periodic')
        [V,D] = eigs(L,5), plot(V)
\end{verbatim}
\bottomspace\par}

\begin{figure}
\begin{center}
\includegraphics[scale=.55]{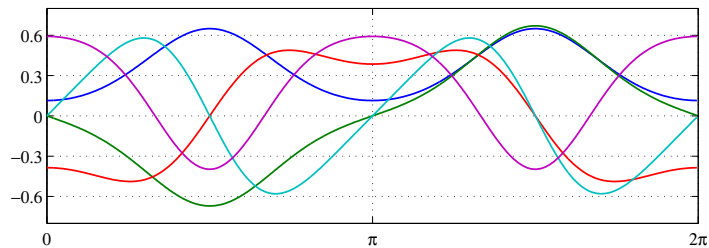}
\end{center}
\caption{\label{mathieu} First five eigenfunctions of the Mathieu
equation $(\ref{mathieueq})$ with $q=2$, computed with {\tt eigs}.}
\end{figure}

So far as we know, Chebfun is the only system offering
such convenient solution of ODEs and related
problems, now in the periodic as well as nonperiodic case.

We have also implemented a periodic analogue of Chebfun's
{\tt expm} command for computing
exponentials of linear operators, which
we omit discussing here for reasons of space.  All the capabilities
mentioned in this section can be explored with Chebgui, the graphical
user interface written by Birkisson, which now invokes
trigonometric spectral discretizations when periodic boundary
conditions are specified.

\section{Discussion}
Chebfun is an open-source project written in MATLAB and
hosted on GitHub; details and the user's guide can be found at
{\tt www.chebfun.org}~\cite{chebbook}.
About thirty people have contributed to
its development over the years, and at present there are about
ten developers based mainly at the University of Oxford.
During 2013--2014 the code was redesigned and rewritten as
version 5 (first released June 2014) in
the form of about 100,000 lines of code realizing
about 40 classes.  The aim of this redesign was to enhance Chebfun's
modularity, clarity, and extensibility, and the introduction
of periodic capabilities, which had not been planned in
advance, was the first big test of this extensibility.
We were pleased to
find that the modifications proceeded smoothly.
The central new feature is a new
class {\tt @trigtech} in parallel to the existing
{\tt @chebtech1} and {\tt @chebtech2}, which work with
polynomial interpolants in first- and second-kind Chebyshev points,
respectively.

About half the classes of Chebfun are concerned with representing
functions, and the remainder are mostly 
concerned with ODE discretization and automatic
differentiation for solution of nonlinear problems, whether
scalar or systems, possibly with nontrivial block structure.
The incorporation of periodic problems into this second, more
advanced part of Chebfun was achieved by introducing
a new class {\tt @trigcolloc} matching
{\tt @chebcolloc1} and {\tt @chebcolloc2}.

About a dozen software projects in various computer languages have been
modeled on Chebfun, and a partial
list can be found at {\tt www.chebfun.org}.\ \ One of
these, Fourfun, is a MATLAB system for periodic functions
developed independently of the
present work by Kristyn McLeod, a student of former Chebfun
developer Rodrigo Platte~\cite{mcleod}.  Another that also has
periodic and differential equations
capabilities is ApproxFun, written in Julia by Sheehan
Olver and former Chebfun developer
Alex Townsend~\cite{julia}.\footnote{Platte created
Chebfun's edge detection algorithm for fast splitting
of intervals.  Townsend extended Chebfun
to two dimensions.}  We think the enterprise of
numerical computing with functions is here to stay,
but cannot predict what systems or languages may be dominant,
say, twenty years from now.  For the moment, only Chebfun offers
the breadth of capabilities entailed in the vision of
MATLAB-like functionality for continuous functions and operators in
analogy to the long-familiar methods for discrete
vectors and matrices.

In this article we
have not discussed Chebfun computations with two-dimensional
periodic functions,
which are under development.  For example, we are
investigating capabilities for solution of time-dependent PDEs on a
periodic spatial domain and for PDEs in two space dimensions, one or
both of which are periodic.  A particularly interesting prospect is
to apply such representations to computation with functions
on disks and spheres. 

For computing with vectors and matrices, although MATLAB codes are
rarely the fastest in execution, their convenience makes them
nevertheless the best tool for many applications.
We believe that Chebfun, including now its extension to periodic problems,
plays the same role for numerical computing with functions.

\section*{Acknowledgements}
This work was carried out in collaboration with the rest of
the Chebfun team, whose names are listed at
{\tt www.chebfun.org}.\ \ Particularly active in this phase of the project have 
been Anthony Austin, \'Asgeir Birkisson,
Toby Driscoll, Nick Hale, Hrothgar (an Oxford graduate
student who has just a single name),
Alex Townsend, and Kuan Xu.
We are grateful to all of these people for their suggestions
in preparing this paper. 
The first author would like to thank the Oxford University Mathematical
Institute, and in particular the Numerical Analysis Group, for hosting
and supporting his sabbatical visit in
2014, during which this research was initiated.

\end{document}